\documentclass[leqno,12pt]{amsart} 
\usepackage{amssymb}
\theoremstyle{plain} 
\newtheorem{theorem}{\indent\sc Theorem}[section]
\newtheorem{lemma}[theorem]{\indent\sc Lemma}
\newtheorem{corollary}[theorem]{\indent\sc Corollary}
\newtheorem{proposition}[theorem]{\indent\sc Proposition}

\theoremstyle{definition} 

\newtheorem{remark}[theorem]{\indent\sc Remark}

\def\C{{\mathbf{C}}}
\def\R{{\mathbf{R}}}
\def\H{{\mathbf{H}}}
\def\Z{{\mathbf{Z}}}
\def\D{{\mathbf{D}}}
\def\Pi{{\mathbf{P}}}
\def\Si{{\mathbf{S}}}
\def\RC{{\overline{\mathbf{C}}}} 

\begin{document}

\title[Complete minimal surfaces in Euclidean 4-space]{Remarks on the Gauss images of complete minimal surfaces in Euclidean four-space}

\author[R. Aiyama]{Reiko Aiyama} 

\author[K. Akutagawa]{Kazuo Akutagawa} 

\author[S. Imagawa]{Satoru Imagawa} 

\author[Y. Kawakami]{Yu Kawakami} 

\dedicatory{Dedicated to Professor Hiroo Naitoh on his 65th birthday}

\renewcommand{\thefootnote}{\fnsymbol{footnote}}
\footnote[0]{2010\textit{ Mathematics Subject Classification}.
Primary 53A10; Secondary 30D35, 53C42.}
\keywords{ 
Gauss map, exceptional values, minimal surface, minimal Lagrangian surface. 
}
\thanks{
The second author is supported in part by the Grant-in-Aid for Scientific Research (B), 
No. 24340008, Japan Society for the Promotion of Science.}
\thanks{ 
The fourth author is supported by the Grant-in-Aid for Scientific Research (C), 
No. 15K04840, Japan Society for the Promotion of Science.}
\address{
Institute of Mathematics, \endgraf
University of Tsukuba, \endgraf
Tsukuba, 305-8571, Japan
}
\email{aiyama@math.tsukuba.ac.jp}

\address{
Department of Mathematics, \endgraf
Tokyo Institute of Technology, \endgraf
Tokyo, 152-8551, Japan
}
\email{akutagawa@math.titech.ac.jp}

\address{
Graduate School of Natural Science and Technology, \endgraf
Kanazawa university, \endgraf
Kanazawa, 920-1192, Japan
}
\email{satoru4105@yahoo.co.jp}

\address{
Faculty of Mathematics and Physics, \endgraf
Institute of Science and Engineering, \endgraf 
Kanazawa University, \endgraf
Kanazawa, 920-1192, Japan
}
\email{y-kwkami@se.kanazawa-u.ac.jp}

\maketitle

\begin{abstract}
We perform a systematic study of the image of the Gauss map for complete minimal surfaces 
in Euclidean four-space. In particular, we give a geometric interpretation of the maximal number 
of exceptional values of the Gauss map of a complete orientable minimal surface in Euclidean four-space. 
We also provide optimal results for the maximal number of exceptional values of the Gauss map of a complete minimal 
Lagrangian surface in the complex two-space and the generalized Gauss map of a complete nonorientable minimal surface 
in Euclidean four-space. 
\end{abstract}

\section{Introduction}\label{section1}

The study of geometric aspects of value distribution theory of complex analytic mappings has achieved many important 
advances. One of the most brilliant results in the study is to give a geometric interpretation of the precise maximum 
for the number of exceptional values of a nonconstant holomorphic map from the complex plane $\C$ to a closed Riemann surface 
$\overline{\Sigma}_{\gamma}$ of genus $\gamma$. Here we call a value that a function or a map never attains an 
{\it exceptional value} of the function or map. In fact, Ahlfors \cite{Ah1935} and Chern \cite{Ch1960} proved that 
the least upper bound for the number of exceptional values of a nonconstant holomorphic map from $\C$ to 
$\overline{\Sigma}_{\gamma}$ coincides with the Euler characteristic of $\overline{\Sigma}_{\gamma}$ by using 
Nevanlinna theory (see also \cite{Ko2003, NO1990, NW2014, Ru2001}). In particular, for a nonconstant meromorphic function on 
$\C$, the geometric interpretation of the maximal number $2$ of exceptional values is the Euler characteristic of the 
Riemann sphere $\RC :=\C\cup \{\infty \}$. We remark that if the closed Riemann surface is of $\gamma \ge 2$, then 
such a map does not exist because the Euler characteristic is negative. 

There exist several classes of immersed surfaces in $3$-dimensional space forms whose Gauss maps have 
value-distribution-theoretical property. For instance, Fujimoto \cite[Theorem I]{Fu1988} proved that the Gauss map of a nonflat 
complete minimal surface in Euclidean $3$-space ${\R}^{3}$ can omit at most $4$ values. The fourth author and Nakajo \cite{KN2012} 
obtained that the maximal number of exceptional values of the Lagrangian Gauss map of a weakly complete improper affine 
front in the affine $3$-space ${\R}^{3}$ is $3$, unless it is an elliptic paraboloid. We note that an improper affine front 
is also called an improper affine map in \cite{MA2005}. We here call it an improper affine front because Nakajo \cite{Na2009} 
and Umehara and Yamada \cite{UY2011} showed that an improper affine map is a front in ${\mathbf{R}}^{3}$. 
Moreover, we \cite{Ka2014} gave 
similar result for flat fronts in ${\H}^{3}$. In \cite{Ka2013}, we obtained a geometric interpretation for the maximal number of 
exceptional values of their Gauss maps. To be precise, we gave a curvature bound for the conformal metric 
$ds^{2}=(1+|g|^{2})^{m}|\omega|^{2}$ on an open Riemann surface $\Sigma$, where $m$ is a positive integer, $\omega$ is a 
holomorphic $1$-form and $g$ is a meromorphic function on $\Sigma$ (\cite[Theorem 2.1]{Ka2013}) and, as a corollary of 
the theorem, proved that the precise maximal number of exceptional values of the nonconstant meromorphic function $g$ on 
$\Sigma$ with the complete conformal metric $ds^{2}$ is $m+2$ (\cite[Corollary 2.2 and Proposition 2.4]{Ka2013}). 
We note that the geometric meaning of the $2$ in $m+2$ is the Euler characteristic of $\RC$ (\cite[Remark 2.3]{Ka2013}). 
Since the induced metric from ${\R}^{3}$ of a minimal surface is $ds^{2}=(1+|g|^{2})^{2}|\omega|^{2}$ (i.e., $m=2$), 
the maximal number of exceptional values of the Gauss map $g$ of a nonflat complete minimal surface in ${\R}^{3}$ is 
$4\,(=2+2)$. For the Lagrangian Gauss map $\nu$ of a weakly complete improper affine front, because $\nu$ is 
meromorphic, $dG$ is holomorphic and the complete metric $d{\tau}^{2}=(1+|\nu|^{2})|dG|^{2}$ (i.e., $m=1$), the maximal 
number of exceptional values of the Lagrangian Gauss map of a weakly complete improper affine front is $3\,(=1+2)$, 
unless it is an elliptic paraboloid.  

On the other hand, Fujimoto \cite[Theorem I\hspace{-.1em}I]{Fu1988} also obtained an optimal estimate for 
the number of exceptional values of the Gauss map of a nonflat complete (orientable) minimal surface in ${\R}^{4}$, and 
Hoffman and Osserman \cite{HO1980} gave a similar result for a nonflat algebraic minimal surface in ${\R}^{4}$ 
(by algebraic minimal surface, we mean a complete minimal surface with finite total curvature). Recently, 
we \cite{Ka2009} gave an effective estimate for 
the number of exceptional values of the Gauss map for a special class of complete minimal surfaces in ${\R}^{4}$ that includes 
algebraic minimal surfaces (this class is called the pseudo-algebraic minimal surfaces. For the corresponding result 
in ${\R}^{3}$, see \cite{KKM2008}). This also provided a geometric interpretation of the Fujimoto and Hoffman-Osserman 
results for this class, because the estimate is described in terms of geometric invariants. 
However, from \cite{Ka2009}, it was still not possible to understand a geometric interpretation for general class. 
Moreover there has been no unified explanation for the study of the image of the Gauss map of complete minimal surfaces 
in ${\R}^{4}$ including nonorientable case. 

The purpose of this paper is to perform a systematic study of the image of the Gauss map for complete minimal 
surfaces in ${\R}^{4}$. The paper is organized as follows: In Section \ref{section2}, we give an optimal estimate 
for the size of the image of the holomorphic map $G=(g_{1}, \ldots, g_{n})\colon \Sigma \to (\RC)^{n}:=
\underbrace{\RC\times \cdots \times \RC}_{n}$ on an open Riemann surface $\Sigma$ with the complete conformal metric 
$
ds^{2}= \prod_{i=1}^{n}(1+|g_{i}|^{2})^{m_{i}}|\omega|^{2}, 
$
where $\omega$ is a holomorphic $1$-form on $\Sigma$ and each $m_{i}$ $(i=1, \cdots, n)$ is a positive integer 
(Theorem \ref{thm-main} and Proposition \ref{prop-main}). 
The result is a generalization of \cite[Corollary 2.2]{Ka2013}. 
In Section \ref{section3.1}, applying the result, we give a geometric interpretation 
of the Fujimoto result \cite[Theorem I\hspace{-.1em}I]{Fu1988} for the maximal number of exceptional values of the Gauss map 
$G=(g_{1}, g_{2})$ of a complete orientable minimal surface in ${\R}^{4}$, that is, the maximal number deeply depends 
on the induced metric from ${\R}^{4}$ and the Euler characteristic of ${\RC}$. In Section \ref{section3.2}, 
after reviewing basic facts, we give the maximal number of exceptional values of the nonconstant part of the Gauss map of 
a complete minimal Lagrangian surface in ${\C}^{2}$ (Corollary \ref{thm-appl-3}). 
In Section \ref{section3.3}, we study the value distribution of the generalized Gauss map of a complete nonorientable minimal 
surface in ${\R}^{4}$. Recently the study of complete nonorientable minimal surfaces has attracted a lot of attention (for example, 
see \cite{AL2015}, \cite{AFL2016}, \cite{LMM2006}, \cite{Ro2006}, \cite{Ro1992} and \cite{Ro1997}, for a good survey see 
\cite{Ma2005}). In \cite{FL2010}, the geometry and topology of complete maximal surfaces with lightlike singularities in 
the Lorentz-Minkowski $3$-space are studied. In this paper, we give an effective estimate for the maximal number of 
exceptional values of the generalized 
Gauss map of a complete nonorientable minimal surface in ${\R}^{4}$ (Corollary \ref{thm-appl-nonori-1}). 
Moreover, by using the argument of L\'opez-Mart\'in 
\cite{LM2000}, we construct examples showing that the estimate is shrap (Proposition \ref{thm-appl-nonori-2} and 
Remark \ref{rmk-appl-nonori-2}). 

\section{Main theorem}\label{section2}
We first state the main theorem of this paper. 

\begin{theorem}\label{thm-main}
Let $\Sigma$ be an open Riemann surface with the conformal metric 
\begin{equation}\label{equ-conformal}
ds^{2}=\displaystyle \prod_{i=1}^{n}(1+|g_{i}|^{2})^{m_{i}}|\omega|^{2}, 
\end{equation}
where $G=(g_{1}, \ldots , g_{n})\colon \Sigma \to (\RC)^{n}:=
\underbrace{\RC\times \cdots \times \RC}_{n}$ 
is a holomorphic map, $\omega$ is a holomorphic 
$1$-form on $\Sigma$ and each $m_{i}$ $(i=1, \cdots, n)$ is a positive integer. Assume that $g_{i_{1}}, \ldots, g_{i_{k}}$
$(1\leq i_{1}< \cdots <i_{k} \leq n)$ are nonconstant and the others are constant. If the metric $ds^{2}$ is complete and 
each $g_{i_{l}}$ $(l=1, \cdots , k)$ omits $q_{i_{l}}> 2$ distinct values, then we have 
\begin{equation}\label{equ-exc}
\displaystyle \sum_{l=1}^{k} \dfrac{m_{i_{l}}}{q_{i_{l}}-2}\geq 1. 
\end{equation}
\end{theorem}

We note that Theorem \ref{thm-main} also holds for the case where at least one of $m_{1}, \ldots, m_{n}$ is positive and 
the others are zeros. For instance, we assume that $g:= g_{i_{1}}$ is nonconstant and the others are constant. 
If $m:=m_{i_{1}}$ is a positive integer and the others are zeros, then the inequality (\ref{equ-exc}) coincides with 
$$
\dfrac{m}{q-2}\geq 1 \, \Longleftrightarrow  \, q \leq m+2,  
$$
where $q:=q_{i_{1}}$. The result corresponds with \cite[Corollary 2.2]{Ka2013}. 
Moreover if all $m_{i}$ are zeros, then the metric $ds^{2}=|\omega|^{2}$ is flat and complete on $\Sigma$. 
We thus may assume that each $g_{i_{l}}$ is a nonconstant meromorphic function on $\C$ 
because there exists a holomorphic universal covering map $\pi \colon \C \to \Sigma$ and 
each $g_{i_{l}}$ is replaced by $g_{i_{l}}\circ \pi$. By the little Picard theorem, 
we have that each $g_{i_{l}}$ can omit at most $2$ distinct values. 
We remark that the geometric interpretation of the precise maximum $2$ for the number of exceptional values 
of a nonconstant meromorphic function on $\C$ is the Euler characteristic of the Riemann sphere $\RC$ 
(\cite{Ah1935}, \cite{Ch1960}). 

The inequality (\ref{equ-exc}) is optimal because there exist the following examples. 

\begin{proposition}\label{prop-main}
Let $\Sigma$ be the complex plane punctured at $p-1$ distinct points ${\alpha}_{1}, \ldots, {\alpha}_{p-1}$ or 
the universal cover of that punctured plane. We set 
$$
\omega =\dfrac{dz}{\prod_{j=1}^{p-1}(z-{\alpha}_{j})} 
$$
and the map $G=(g_{1}, \ldots , g_{n})$ is given by 
$$
g_{i_{1}}= \cdots = g_{i_{k}}=z \quad (1\leq i_{1}< \cdots < i_{k}\leq n )
$$ 
and the others are constant. Then all $g_{i_{l}}$ $(l=1, \cdots, k)$ omit $p$ distinct values 
${\alpha}_{1}, \ldots, {\alpha}_{p-1}, \infty$ and the metric (\ref{equ-conformal}) is complete if and only if 
$$
p\leq 2+\displaystyle \sum_{l=1}^{k}m_{i_{l}}. 
$$
In particular, there exist examples which satisfy the equality of (\ref{equ-exc}). 
\end{proposition}
\begin{proof}
A divergent path $\Gamma$ in $\Sigma$ must tend to one of the points ${\alpha}_{1}, \ldots, {\alpha}_{p-1}$ or $\infty$. 
Thus we have 
$$
\int_{\Gamma} ds= \int_{\Gamma}\, \prod_{i=1}^{n}(1+|g_{i}|^{2})^{m_{i}/2}|\omega| 
= C \int_{\Gamma} \dfrac{\prod_{l=1}^{k}(1+|z|^{2})^{m_{i_{l}}/2}}{\prod_{j=1}^{p-1}|z-{\alpha}_{j}|}|dz|= \infty 
$$
when $p\leq 2+\sum_{l=1}^{k}m_{i_{l}}$. Here $C$ is some constant. Then the equality of (\ref{equ-exc}) holds if and only if 
$p=2+\sum_{l=1}^{k}m_{i_{l}}$. 
\end{proof}

Before proceeding to the proof of Theorem \ref{thm-main}, we recall the notion of chordal distance between 
two distinct values in $\RC$ and two function-theoretic lemmas. For two distinct values 
$\alpha$, $\beta\in \RC$, we set 
$$
|\alpha, \beta|:= \dfrac{|\alpha -\beta|}{\sqrt{1+|\alpha|^{2}}\sqrt{1+|\beta|^{2}}}
$$
if $\alpha \not= \infty$ and $\beta \not= \infty$, and $|\alpha, \infty|=|\infty, \alpha| := 1/\sqrt{1+|\alpha|^{2}}$. 
We note that, if we take $v_{1}$, $v_{2}\in {\Si}^{2}$ with $\alpha =\varpi (v_{1})$ and $\beta = \varpi (v_{2})$, we have that 
$|\alpha, \beta|$ is a half of the chordal distance between $v_{1}$ and $v_{2}$, where $\varpi$ denotes the stereographic projection of 
the $2$-sphere ${\Si}^{2}$ onto $\RC$. 

\begin{lemma}{\cite[(8.12) in page 136]{Fu1997}}\label{Lem-main1}
Let $g$ be a nonconstant meromorphic function on ${\Delta}_{R}=\{z\in \C ; |z|< R \}$ $(0<R\leq +\infty)$ 
which omits $q$ values 
${\alpha}_{1}, \ldots, {\alpha}_{q}$. If $q>2$, then for each positive $\eta$ with $\eta <(q-2)/q$, 
there exists a positive constant $C'$ 
depending on $q$ and $L:=\min_{i< j}|{\alpha}_{i}, {\alpha}_{j}|$ such that 
\begin{equation}\label{equ-lemma1}
\dfrac{|g'_{z}|}{(1+|g|^{2})\prod_{j=1}^{q}|g, {\alpha}_{j}|^{1-\eta}}\leq C'\dfrac{R}{R^{2}-|z|^{2}}. 
\end{equation}  
\end{lemma}

\begin{lemma}{\cite[Lemma 1.6.7]{Fu1993}}\label{Lem-main2} 
Let $d{\sigma}^{2}$ be a conformal flat-metric on an open Riemann surface $\Sigma$. 
Then, for each point $p\in \Sigma$, there exists a local diffeomorphism $\Phi$ of a 
disk ${\Delta}_{R}=\{z\in \C ; |z|< R \}$ $(0<R\leq +\infty)$ onto an open 
neighborhood of $p$ with $\Phi (0)=p$ such that $\Phi$ is an isometry, that is, 
the pull-back ${\Phi}^{\ast}(d{\sigma}^{2})$ is equal to the standard Euclidean metric $ds^{2}_{E}$ on ${\Delta}_{R}$ 
and that, for a specific point $a_{0}$ with $|a_{0}|=1$, the ${\Phi}$-image ${\Gamma}_{a_{0}}$ of 
the curve $L_{a_{0}}=\{w:= a_{0}s ; 0 < s < R\}$ 
is divergent in $\Sigma$. 
\end{lemma}

\begin{proof}[{\it Proof of Theorem \ref{thm-main}}] 
Assume that each $g_{i_{l}}$ ($l=1, \cdots, k$) omits 
$q_{i_{l}}$ distinct values, ${\alpha}_{1}^{l}, \ldots, {\alpha}_{q_{i_{l}}}^{l}$. 
After a suitable M\"obius transformation for each $g_{i_{l}}$, we may assume that 
${\alpha}_{q_{i_{1}}}^{1}=\cdots ={\alpha}_{q_{i_{k}}}^{k}=\infty$. 
Suppose that each $q_{i_{l}}> 2$ and 
\begin{equation}\label{equ-main-proof-1}
\displaystyle \sum_{l=1}^{k} \dfrac{m_{i_{l}}}{q_{i_{l}}-2}< 1. 
\end{equation}
Then, by (\ref{equ-main-proof-1}), we ultimately suppose that $q_{i_{l}}> m_{i_{l}}+2$ for each $i_{l}$ $(l=1, \cdots, k)$. 
Taking some positive number $\eta$ with 
\begin{equation}\label{equ-main-proof-2}
0 < \eta < \dfrac{q_{i_{l}}-2-m_{i_{l}}}{q_{i_{l}}}
\end{equation}
for each $i_{l}$ ($l=1, \cdots, k$). We set 
$$
\lambda_{i_{l}}:= \dfrac{m_{i_{l}}}{q_{i_{l}}-2-q_{i_{l}}\eta} \quad (l=1, \cdots, k). 
$$
For a sufficiently small number $\eta$, we have 
\begin{equation}\label{equ-main-proof3}
\displaystyle \Lambda := \sum_{l=1}^{k} {\lambda}_{i_{l}} = \sum_{l=1}^{k} \dfrac{m_{i_{l}}}{q_{i_{l}}-2-q_{i_{l}}\eta} <1
\end{equation}
and 
\begin{equation}\label{equ-main-proof4}
\dfrac{{\lambda}_{i_{l}}}{1-\Lambda}> 1 \quad (l=1, \cdots, k). 
\end{equation} 
Then we define a new metric 
\begin{equation}\label{equ-main-proof5}
\displaystyle d{\sigma}^{2}=|\hat{\omega}_{z}|^{\frac{2}{1-\Lambda}} \prod_{l=1}^{k} 
\Biggl{(}\dfrac{1}{|g'_{i_{l}}|}\prod_{j=1}^{q_{i_{l}}-1} 
\biggl{(}\dfrac{|g_{i_{l}}-{\alpha}_{j}^{l}|}{\sqrt{1+|{\alpha}_{j}^{l}|^{2}}} \biggr{)}^{1-\eta}\Biggr{)}^{\frac{2{\lambda}_{i_{l}}}{1-\Lambda}} |dz|^{2}
\end{equation}
on ${\Sigma}'=\{p\in \Sigma \,;\, g'_{i_{l}}\not= 0 \;\text{for each $l$}\}$, where $\omega =\hat{\omega}_{z}dz$ and $g'_{i_{l}}=dg_{i_{l}}/dz$. 
Take a point $p\in {\Sigma}'$. Since $d{\sigma}^{2}$ is flat, by Lemma \ref{Lem-main2}, there exists an isometry $\Phi$ satisfying $\Phi (0)= p$ from 
a disk $\triangle_{R}=\{z\in \C \,;\, |z|<R \}$ $(0< R\leq +\infty)$ with the standard Euclidean metric $ds^{2}_{E}$ 
onto an open neighborhood of $p\in {\Sigma}'$ 
with the metric $d{\sigma}^{2}$, such that, for a specific point $a_{0}$ with $|a_{0}|=1$, the $\Phi$-image ${\Gamma}_{a_{0}}$ of the curve 
$L_{a_{0}}=\{w=a_{0}s \,;\, 0<s<R \}$ is divergent in ${\Sigma}'$. For brevity, we denote $g_{i_{l}}\circ \Phi$ on $\triangle_{R}$ by $g_{i_{l}}$ in the following. 
By Lemma \ref{Lem-main1}, for each $i_{l}$, we get 
\begin{equation}\label{equ-main-proof6}
\displaystyle R\leq C'_{i_{l}}\dfrac{1+|g_{i_{l}}(0)|^{2}}{|g'_{i_{l}}(0)|} \prod_{j=1}^{q_{i_{l}}}|g_{i_{l}}(0), {\alpha}_{j}^{l}|^{1-\eta} < +\infty,  
\end{equation}
that is, the radius $R$ is finite. Hence 
$$
L_{d\sigma} (\Gamma_{a_{0}}) =\int_{\Gamma_{a_{0}}} d\sigma = R <+\infty , 
$$ 
where $L_{d\sigma} (\Gamma_{a_{0}})$ denotes the length of $\Gamma_{a_{0}}$ with respect to the metric $d{\sigma}^{2}$. 

Now we prove that ${\Gamma}_{a_{0}}$ is divergent in $\Sigma$. Indeed, if not, then ${\Gamma}_{a_{0}}$ must tend to a point 
$p_{0}\in \Sigma \backslash {\Sigma}'$, where $g'_{i_{l}}(p_{0})= 0$ for some $i_{l}$. Taking a local complex coordinate 
$\zeta :=g'_{i_{0}}$ in a neighborhood of $p_{0}$ with $\zeta (p_{0})=0$, we can write the metric $d{\sigma}^{2}$ as 
$$
d{\sigma}^{2} = |\zeta|^{-2{\lambda}_{i_{l}}/(1-\Lambda )}\, w\, |d\zeta|^{2}, 
$$
with some positive function $w$. Since ${\lambda}_{i_{l}}/(1-\Lambda) > 1$, we have 
$$
R= \int_{{\Gamma}_{a_{0}}} d\sigma > \widetilde{C} \int_{{\Gamma}_{a_{0}}} \dfrac{|d\zeta|}{|\zeta|^{{\lambda}_{i_{l}}/(1-\Lambda )}} = +\infty. 
$$
Moreover, in the same way, if there exists a subset $\{ l_{1}, \ldots , l_{m}\}$ in $\{ 1, \cdots, k \}$ such that 
each $g_{i_{l_{j}}}$ $(j=1, \cdots, m)$ have a zero at $p_{0}$, we also get that $R= +\infty$ because 
$$
\displaystyle \sum_{s=1}^{m} \dfrac{{\lambda}_{i_{l_{s}}}}{1-\Lambda} >1. 
$$
These contradict that $R$ is finite. 

Since ${\Phi}^{\ast}d{\sigma}^{2}=|dz|^{2}$, we have by (\ref{equ-main-proof5}) that 
$$
\displaystyle |\hat{\omega}_{z}|= \prod_{l=1}^{k} \Biggl{(}|g'_{i_{l}}| \prod_{j=1}^{q_{i_{l}}-1} \biggl{(}\dfrac{\sqrt{1+|{\alpha}_{j}^{l}|^{2}}}{|g_{i_{l}}-{\alpha}_{j}^{l}|} \biggr{)}^{1-\eta}\Biggr{)}^{{\lambda}_{i_{l}}}. 
$$
By Lemma \ref{Lem-main1}, 
we have 
\begin{eqnarray*}
{\Phi}^{\ast} ds &=& |\hat{\omega}_{z}|\prod_{i=1}^{n} (1+|g_{i}|^{2})^{m_{i}/2} |dz| \\
                 &\leq & C_{1} \Biggl{(}\prod_{l=1}^{k}|g'_{i_{l}}| (1+|g_{i_{l}}|^{2})^{m_{i_{l}}/2\lambda_{i_{l}}}\prod_{j=1}^{q_{i_{l}}-1} \Biggl{(}\dfrac{\sqrt{1+|{\alpha}_{j}^{l}|^{2}}}{|g_{i_{l}}-{\alpha}_{j}^{l}|} \Biggr{)}^{1-\eta} \Biggr{)}^{\lambda_{i_{l}}} |dz| \\
                                   &=& C_{1}\prod_{l=1}^{k} \Biggl{(}\dfrac{|g'_{i_{l}}|}{(1+|g_{i_{l}}|^{2})\prod_{j=1}^{q_{i_{l}}}|g_{i_{l}}, {\alpha}^{l}_{j}|^{1-\eta}} \Biggr{)}^{\lambda_{i_{l}}} |dz| \leq C_{2}\Biggl{(}\dfrac{R}{R^{2}-|z|^{2}} \Biggr{)}^{\Lambda} |dz|. 
\end{eqnarray*}
Now we consider the geodesic distance $d(p)$ with the respect to the metric $ds^{2}$ from each point $p\in \Sigma$ to the boundary of $\Sigma$. 
Then we have 
$$
\displaystyle d(p)\leq \int_{{\Gamma}_{a_{0}}} ds = \int_{{L}_{a_{0}}} {\Phi}^{\ast} ds\leq 
C_{2} \int_{{L}_{a_{0}}}\Biggl{(}\dfrac{R}{R^{2}-|z|^{2}} \Biggr{)}^{\Lambda} |dz|\leq C_{2} \dfrac{R^{1-\Lambda}}{1-\Lambda} < +\infty 
$$
because $0< \Lambda < 1$. This contradicts the assumption that the metric $ds^{2}$ is complete. 
\end{proof}

\section{Applications}\label{section3}

\subsection{Gauss images of complete orientable minimal surfaces in ${\R}^{4}$}\label{section3.1}
We first recall some basic facts of minimal surfaces in ${\R}^{4}$. Details can be found, for example, 
\cite{Ch1965, HO1980, HO1985, Os1964}. Let $X=(x^{1}, x^{2}, x^{3}, x^{4})\colon \Sigma \to {\R}^{4}$ be an oriented minimal 
surface in ${\R}^4$. By associating a local complex coordinate $z=u+\sqrt{-1}v$ with each positive isothermal coordinate system 
$(u, v)$, $\Sigma$ is considered as a Riemann surface whose conformal metric is the induced metric $ds^{2}$ from ${\R}^{4}$. 
Then 
\begin{equation}\label{equ-appl-min-1}
\triangle_{ds^{2}} X=0
\end{equation}
holds, that is, each coordinate function $x^{i}$ is harmonic. With respect to the local coordinate $z$ of the surface, 
(\ref{equ-appl-min-1}) is given by 
$$
\bar{\partial} \partial X =0, 
$$
where $\partial =(\partial /\partial u - \sqrt{-1}\partial /\partial v)/2$, $\bar{\partial} 
=(\partial /\partial u + \sqrt{-1}\partial /\partial v)/2$. Hence each ${\phi}_{i}:= \partial x^{i} dz$ ($i=1, 2, 3, 4$) is a 
holomorphic $1$-form on $\Sigma$. If we set that 
$$
\omega = {\phi}_{1} -\sqrt{-1} {\phi}_{2}, \qquad g_{1}=\dfrac{{\phi}_{3}+\sqrt{-1}{\phi}_{4}}{{\phi}_{1} -\sqrt{-1} {\phi}_{2}}, 
\qquad g_{2}=\dfrac{-{\phi}_{3}+\sqrt{-1}{\phi}_{4}}{{\phi}_{1} -\sqrt{-1} {\phi}_{2}}, 
$$
then $\omega$ is a holomorphic $1$-form and $g_{1}$ and $g_{2}$ are meromorphic functions on $\Sigma$. 
Moreover the holomorphic map $G:=(g_{1}, g_{2})\colon \Sigma \to \RC \times \RC$ coincides with the Gauss map of $X(\Sigma)$. 
We remark that the Gauss map of $X(\Sigma)$ in ${\R}^{4}$ is the map from each point of $\Sigma$ to its oriented tangent plane, 
the set of all oriented (tangent) planes in ${\R}^{4}$ is naturally identified with the quadric 
$$
\mathbf{Q}^{2}(\C) =\{[w^{1}: w^{2}: w^{3}: w^{4}] \in \mathbf{P}^{3}(\C) \, ;\, (w^{1})^{2}+\cdots +(w^{4})^{2} = 0\}
$$
in $\mathbf{P}^{3}(\C)$, and the quadric $\mathbf{Q}^{2}(\C)$ is biholomorphic to the product of the Riemann spheres $\RC \times \RC$. 
Furthermore the induced metric from ${\R}^{4}$ is given by 
\begin{equation}\label{equ-appl-min-2}
ds^{2}= (1+|g_{1}|^{2})(1+|g_{2}|^{2})|\omega|^{2}. 
\end{equation}

Applying Theorem \ref{thm-main} to the metric $ds^{2}$, we can get the Fujimoto theorem for the Gauss map of complete orientable 
minimal surfaces in ${\R}^{4}$. 

\begin{theorem}\cite[Theorem I\hspace{-.1em}I]{Fu1988}\label{thm-appl-1}
Let $X\colon \Sigma \to {\R}^{4}$ be a complete orientable nonflat minimal surface and 
$G=(g_{1}, g_{2})\colon \Sigma \to \RC \times \RC $ the Gauss map of $X(\Sigma)$. 
\begin{enumerate}
\item[(i)] Assume that $g_{1}$ and $g_{2}$ are both nonconstant and omit $q_{1}$ and $q_{2}$ distinct values respectively. 
If $q_{1}> 2$ and $q_{2}> 2$, then we have 
\begin{equation}\label{equ-appl-min-3}
\dfrac{1}{q_{1}-2}+\dfrac{1}{q_{2}-2}\geq 1. 
\end{equation}
\item[(i\hspace{-.1em}i)] If either $g_{1}$ or $g_{2}$, say $g_{2}$, is constant, then $g_{1}$ can omit at most 
$3$ distinct values.  
\end{enumerate}
\end{theorem}
\begin{proof}
We first show (i). Since $g_{1}$ and $g_{2}$ are both nonconstant and $m_{1}=m_{2}=1$ from (\ref{equ-appl-min-2}), 
we can prove the inequality (\ref{equ-appl-min-3}) by Theorem \ref{thm-main}. Next we show (i\hspace{-.1em}i). 
If we set that $g_{1}$ omits $q_{1}$ values, then we obtain 
$$
\dfrac{1}{q_{1}-2}\geq 1
$$
from Theorem \ref{thm-main} because $m_{1}=1$. Thus we have $q_{1}\leq 3$. 
\end{proof}

Hence we reveal that the Fujimoto theorem depends on the orders of the factors $(1+|g_{1}|^{2})$ and $(1+|g_{2}|^{2})$ in 
the induced metric from ${\R}^{4}$ and the Euler characteristic of the Riemann sphere $\RC$. 

\subsection{Gauss images of complete minimal Lagrangian surfaces in ${\C}^{2}$}\label{section3.2}
There exists a complex representation for a minimal Lagrangian surface $\Sigma\,(\subset {\C}^{2})$ in terms of 
holomorphic data. On the representation for the surface $\Sigma$, Chen-Morvan \cite{CM1987} proved that there exists an explicit 
correspondence in ${\C}^{2}$ between minimal Lagrangian surfaces and holomorphic curves with a nondegenerate condition. 
Indeed, this correspondence is given by exchanging the orthogonal complex structure $J$ in ${\C}^{2}$ to another one 
on ${\R}^{4}={\C}^{2}$. For the complete case, this result can also be proved from \cite[Theorem I\hspace{-.1em}I]{Mi1984} and 
the well-known fact \cite{HL1982} that any minimal Lagrangian submanifold in ${\C}^{n}$ is stable. 
More generally, H\'elein-Romon \cite{HR2000, HR2002} and the first author \cite{Ai2001, Ai2004} 
proved that every Lagrangian surface $\Sigma$ in ${\C}^{2}$, not necessarily minimal, is represented in terms of 
a plus spinor (or a minus spinor) of the $\text{spin}^{\C}$ bundle 
$(\underline{\C}_\Sigma\oplus \underline{\C}_\Sigma)\oplus (K^{-1}_{\Sigma}\oplus K_{\Sigma})$ satisfying the Dirac equation with 
potential (see \cite[Section\,1]{Ai2004} for details). Here, 
$\underline{\C}_{\Sigma}$ and $K_{\Sigma}$ denote respectively the trivial complex line bundle and the canonical complex line bundle 
of $\Sigma$. Note that the representation in terms of plus spinors in 
$\Gamma (\underline{\C}_{\Sigma}\oplus \underline{\C}_{\Sigma}) = \Gamma (\Sigma\times {\C}^{2})$ given by 
the first author is a natural generalization of the one given by Chen-Morvan. 
Here we remark that the Lagrangian angle of any minimal Lagrangian surface is constant. 
Combining these results, we get the following: 

\begin{theorem}{$($\cite{CM1987}, \cite{Ai2001, Ai2004}$)$}\label{thm-appl-2}
Let $\Sigma$ be a Riemann surface with an isothermal coordinate $z=u+\sqrt{-1}v$ 
around each point. Let $F = (F_{1}, F_{2})\colon \Sigma\to {\C}^{2}$ be a holomorphic map satisfying 
$|S_{1}|^{2}+|S_{2}|^{2}\not= 0$ everywhere on $\Sigma$, where 
$S_{1}:= (F_{2})'_{z} = dF_2/dz$ and $S_{2}:= - (F_{1})'_{z} = - dF_1/dz$. 
Then 
\begin{equation}\label{equ-appl-lag-1}
f=\dfrac{1}{\sqrt{2}}e^{\sqrt{-1}\,\beta/2}(F_{1}-\sqrt{-1}\, \overline{F_{2}}, F_{2}+\sqrt{-1}\, \overline{F_{1}}) 
\end{equation}
is a minimal Lagrangian conformal immersion from $\Sigma$ to ${\C}^{2}$ with constant Lagrangian angle $\beta \in {\R}/2\pi\Z$. 
The induced metric $ds^{2}$ on $\Sigma$ by $f$ and its Gaussian curvature $K_{ds^{2}}$ are respectively given by 
\begin{equation}\label{eq-appl-lag-2}
ds^{2}=(|S_{1}|^{2}+|S_{2}|^{2})|dz|^{2}, \qquad K_{ds^{2}}=-2\dfrac{|S_{1}(S_{2})_{z}-S_{2}(S_{1})_{z}|}{(|S_{1}|^{2}+|S_{2}|^{2})^{3}}.\end{equation}
Conversely, every minimal Lagrangian immersion $f\colon M\to {\C}^{2}$ with constant Lagrangian angle $\beta$ is congruent 
with the one constructed as above. 
\end{theorem}

Set a meromorphic function $g:=-S_{2}/S_{1}$. Then 
$$
G:=(g, e^{\sqrt{-1}\beta})\colon \Sigma \to \RC \times \RC
$$
can be regarded as the Gauss map of $F(\Sigma )$ in ${\R}^{4}={\C}^{2}$ (cf. \cite{HO1980, HO1985}). 
Thus we get the following result.  

\begin{corollary}\label{thm-appl-3}
The first component $g$ of the Gauss map of a complete minimal Lagrangian surface in ${\C}^{2}$ which is not a Lagrangian 
plane can omit at most $3$ values. 
\end{corollary}
\begin{proof}
We assume that $g$ omits $q$ distinct values and set a holomorphic $1$-form $\omega :=S_{1}dz$ on $\Sigma$. 
In terms of the data $(\omega, g)$ of $\Sigma$, the induced metric can be rewritten by $ds^{2}=(1+|g|^{2})|\omega|^{2}$, 
that is, $m_{1}=1$ and $m_{2}=0$. For this case, the first component $g$ of the Gauss map is nonconstant and the second one 
is constant. From Theorem \ref{thm-main}, we obtain that $q\leq 1+2=3$. 
\end{proof}

\subsection{Generalized Gauss images of complete nonorientable minimal surfaces in ${\R}^{4}$}\label{section3.3}
We first summarize some basic facts of nonorientable minimal surfaces in ${\R}^{4}$. 
For more details, we refer the reader to \cite{El1986} and \cite{Ma2005}. 
Let $\widehat{X}\colon \widehat{\Sigma}\to {\R}^{4}$ be a conformal minimal immersion of a nonorientable 
Riemann surface $\widehat{\Sigma}$ in ${\R}^{4}$. If we consider the orientable conformal double 
cover $\pi \colon \Sigma \to \widehat{\Sigma}$, then the composition $X:=\widehat{X}\circ \pi \colon \Sigma \to {\R}^{4}$ 
is a conformal minimal immersion of the orientable Riemann surface $\Sigma$ in ${\R}^{4}$. 
Let $I\colon \Sigma \to \Sigma$ denote the antiholomorphic order two deck transformation associated to the orientable 
cover $\pi \colon \Sigma \to \widehat{\Sigma}$, then $I^{\ast} ({\phi}_{j})=\bar{\phi}_{j}$ $(j=1, \cdots, 4)$ or equivalently, 
\begin{equation}\label{eq-appl-nonori-1}
g_{1}\circ I = -\dfrac{1}{\bar{g_{1}}}, \qquad g_{2}\circ I = -\dfrac{1}{\bar{g_{2}}}, \qquad I^{\ast}\omega = \overline{g_{1}g_{2}\omega}. 
\end{equation}
Conversely, if $(g_{1}, g_{2}, \omega)$ is the Weierstrass data of an orientable minimal surface $X\colon \Sigma \to {\R}^{4}$ and 
$I$ is an antiholomorphic involution without fixed points in $\Sigma$ satisfying (\ref{eq-appl-nonori-1}), then the unique map 
$\widehat{X}\colon \widehat{\Sigma}=\Sigma /\langle I \rangle \to {\R}^{4}$ satisfying that $X=\widehat{X}\circ \pi$ is 
a nonorientable minimal surface in ${\R}^{4}$. 

The fact that $g_{k}\circ I= -(\bar{g_{k}})^{-1}$ $(k=1, 2)$ implies the existence of a map $\hat{g_{k}}\colon \widehat{\Sigma} 
\to \R\Pi^{2}$ satisfying $\hat{g_{k}}\circ \pi = {\pi}_{0} \circ g_{k}$, where 
${\pi}_{0}\colon \RC \to \R\Pi^{2}\equiv \RC /\langle I_{0} \rangle$ is the natural projection and $I_{0}:=-(\bar{z})^{-1}$ is 
the antipodal map of $\RC$. We call the map 
$\widehat{G}=(\hat{g_{1}}, \hat{g_{2}})\colon \widehat{\Sigma} \to \R\Pi^{2}\times \R\Pi^{2}$ 
the {\it generalized Gauss map} of $\widehat{X}(\widehat{\Sigma})$. Applying Theorem \ref{thm-appl-1} to the 
generalized Gauss map, we get the following: 

\begin{corollary}\label{thm-appl-nonori-1}
Let $\widehat{X}\colon \widehat{\Sigma} \to {\R}^{4}$ be a nonflat complete nonorientable minimal surface and 
$\widehat{G}=(\hat{g_{1}}, \hat{g_{2}})$ the generalized Gauss map of $\widehat{X}(\widehat{\Sigma})$. 
\begin{enumerate}
\item[(i)] Assume that $\hat{g}_{1}$ and $\hat{g}_{2}$ are both nonconstant and omit $q_{1}$ and $q_{2}$ distinct 
points in $\R\Pi^{2}$ respectively. If $q_{1}>1$ and $q_{2}> 1$, then 
\begin{equation}\label{eq-appl-nonori-2}
\dfrac{1}{q_{1}-1}+\dfrac{1}{q_{2}-1}\geq 2. 
\end{equation}
\item[(i\hspace{-.1em}i)] If either $\hat{g}_{1}$ or $\hat{g}_{2}$, say $\hat{g}_{2}$, is constant, 
then $\hat{g}_{1}$ can omit at most $1$ point in $\R\Pi^{2}$. 
\end{enumerate}
\end{corollary}

The inequality (\ref{eq-appl-nonori-2}) is optimal because there exist the following examples. 

\begin{proposition}\label{thm-appl-nonori-2}
There exist nonflat complete nonorientable minimal surfaces in ${\R}^{4}$ each of which 
components $\hat{g_{i}}$ ($i=1, 2$) of the generalized Gauss map $\widehat{G}=(\hat{g_{1}}, \hat{g_{2}})$ 
is nonconstant and omits $2$ distinct points in $\R\Pi^{2}$. 
\end{proposition}

\begin{proof}
We take $2$ distinct points $\alpha$, $\beta$ in $\C\backslash \{0\}$ and assume that $\alpha \not= -(\bar{\beta})^{-1}$. 
Let $\Sigma$ be the complex plane punctured at $4$ distinct points $\alpha$, $\beta$, $-(\bar{\alpha})^{-1}$, 
$-(\bar{\beta})^{-1}$. We set that
$$
\check{g}_{1}= z, \qquad \check{g}_{2}= z, \qquad \check{\omega} = \dfrac{dz}{(z-\alpha )(z-\beta )(\bar{\alpha}z+1)(\bar{\beta}z+1)} 
$$
on $\Sigma$. If we define $\check{I}\colon \Sigma \to \Sigma$, $\check{I}(z)=-(\bar{z})^{-1}$, then $\check{I}$ is an 
antiholomorphic involution without fixed points and the following inequalities hold: 
\begin{equation}\label{thm-appl-nonori-3}
{\check{g}}_{1}\circ \check{I}=-\dfrac{1}{\bar{\check{g}}_{1}}, \qquad 
{\check{g}}_{2}\circ \check{I}=-\dfrac{1}{\bar{\check{g}}_{2}}, \qquad 
{\check{I}}^{\ast}\check{\omega} =\overline{\check{g}_{1}\check{g}_{2}\check{\omega}}. 
\end{equation}
Thus if we set 
$$
\check{\phi}_{1}=\dfrac{1}{2}(1+\check{g}_{1}\check{g}_{2})\check{\omega}, \; 
\check{\phi}_{2}=\dfrac{\sqrt{-1}}{2}(1-\check{g}_{1}\check{g}_{2})\check{\omega}, \; 
\check{\phi}_{3}=\dfrac{1}{2}(\check{g}_{1}-\check{g}_{2})\check{\omega}, \; 
\check{\phi}_{4}=-\dfrac{\sqrt{-1}}{2}(\check{g}_{1}+\check{g}_{2})\check{\omega}, 
$$
then we easily show that $\check{I}^{\ast}{\check{\phi}}_{i}= \overline{\check{\phi}}_{i}$ ($i=1, \cdots , 4$). 
Moreover these holomorphic $1$-forms satisfy that $\sum_{i=1}^{4}{\check{\phi}}_{i}^{2}\equiv 0$ and 
$\sum_{i=1}^{4}|{\check{\phi}}_{i}|^{2}$ is a complete conformal metric on $\Sigma$. 

Let $\widetilde{\Sigma}$ be a universal cover 
surface of $\Sigma$. By the uniformization theorem, we may assume that $\widetilde{\Sigma}$ is the unit disk $\D$. 
Let $\pi \colon \D\to \Sigma$ be the conformal universal covering map and $\widetilde{I}$ a lift of $\check{I}$ to $\D$. 
If we set $\tilde{\phi}_{i}:= \pi^{\ast}({\phi}_{i})$, then 
$\widetilde{I}^{\ast} (\tilde{\phi}_{i})= \overline{\tilde{\phi}_{i}}$ ($i=1, \cdots , 4$). 
Since $\check{I}$ is an antiholomorphic involution on $\Sigma$ without fixed points, 
$\widetilde{I}^{2k+1}$ $(k\in \Z)$ is also an antiholomorphic transformation on $\D$ without fixed points. 
From the argument of the proof of Lemma 1 in \cite{LM2000}, $\widetilde{I}^{2k}$ $(k\in \Z\backslash \{0\})$ 
has no fixed points on $\D$, $\langle \widetilde{I}^{2} \rangle \simeq \Z$, and 
$\D / \langle \widetilde{I}^{2} \rangle$ is biholomorphic to the annulus $A(R) =\{z\in\C \,;\, R^{-1}< |z|< R\}$ 
for a suitable $R >1$. Since $(\widetilde{I}^{2})^{\ast}(\tilde{\phi}_{i}) = \tilde{\phi}_{i}$, each 
holomorphic 1-form $\tilde{\phi}_{i}$ ($i=1, \cdots, 4$) can be induced on the quotient 
$\D / \langle \widetilde{I}^{2} \rangle$. The corresponding holomorphic 1-forms on 
$\D / \langle \widetilde{I}^{2} \rangle$ are denoted by ${\phi}_{1}$, ${\phi}_{2}$, ${\phi}_{3}$ and ${\phi}_{4}$, 
and obviously satisfy that $\sum_{i=1}^{4} {\phi}_{i}^{2}\equiv 0$, $ds^{2}:= \sum_{i=1}^{4} |{\phi}_{i}|^{2}$ is 
a complete conformal metric on $\D / \langle \widetilde{I}^{2} \rangle \simeq A(R)$ and 
$I^{\ast}({\phi}_{i})=\bar{{\phi}}_{i}$ ($i=1, \cdots , 4$), where $I\colon A(R)\to A(R)$ induced by $\tilde{I}$. Then 
it holds that $I(z)=-(\bar{z})^{-1}$ on $A(R)$. Moreover the two meromorphic functions 
$$
g_{1}=\dfrac{{\phi}_{3}+\sqrt{-1}{\phi}_{4}}{{\phi}_{1}-\sqrt{-1}{\phi}_{2}} \qquad \text{and} \qquad 
g_{2}=\dfrac{-{\phi}_{3}+\sqrt{-1}{\phi}_{4}}{{\phi}_{1}-\sqrt{-1}{\phi}_{2}}
$$
on $A(R)$ omit 4 points $\alpha$, $\beta$, $-(\bar{\alpha})^{-1}$ and $-(\bar{\beta})^{-1}$ in $\RC$. 

Let $f\colon \RC \to \RC$ be a rational function given in Lemma 2 in \cite{LM2000}, that is, the function $f$ satisfies the following 
three conditions: 
\begin{enumerate}
\item[(a)] The only poles of $f$ are $0$ and $\infty$, 
\item[(b)] $f\circ I_{0}= \bar{f}$, 
\item[(c)] $f$ has no zeros on the circle $\{z \, ;\, |z|=1\}$. 
\end{enumerate}
Set ${\phi}_{j}=({\varphi}_{j}/z) dz$ $(j=1, \cdots , 4)$ and write the Laurent series expansion of ${\varphi}_{j}$ as 
$$
\displaystyle {\varphi}_{j}(z) = a_{0}^{j}+ \sum_{n>0} (a_{n}^{j}z^{n}+ (-1)^{n+1}\bar{a}_{n}^{j}z^{-n}), \quad 
a_{0}^{j} \in \sqrt{-1}\R. 
$$ 
We easily check that the Laurent series expansion of $f$ is written as 
$$
\displaystyle f(z)=\sum_{n=1}^{m} (b_{n}z^{n}+(-1)^{n}\bar{b}_{n}z^{-n}), 
$$
where $m\in \Z_{+}$. Let $k$ be an odd positive number with $k>m$. Then it holds that
\begin{equation}\label{equ-appl-nonori-4}
\text{Res}_{z=0}\Biggl{(}\biggl{[}\sum_{n> 0} (a^{j}_{n}z^{kn}+ (-1)^{n+1}\bar{a}^{j}_{n} z^{-kn}) \biggr{]} f(z)
\dfrac{dz}{z} \Biggr{)}=0, \quad j=1, \cdots , 4. 
\end{equation}
Furthermore, by the virtue of the property for $f(z)$, we have 
\begin{equation}\label{equ-appl-nonori-5}
\text{Res}_{z=0}\Biggl{(}a^{j}_{0} f(z)\dfrac{dz}{z} \Biggr{)} =0, \quad j=1, \cdots , 4. 
\end{equation}

We consider the covering $T_{k}\colon A(R^{1/k})\to A(R)$, $T_{k}(z) =z^{k}$ and  
define the holomorphic $1$-forms $\psi_{j}$ $(j=1, \cdots , 4)$ on $A(R^{1/k})$ as follows: 
$$
{\psi}_{j}:= f(z)T^{\ast}_{k}({\phi}_{j}) = kf(z){\varphi}_{j}(z^{k}) \dfrac{dz}{z}.  
$$
From (\ref{equ-appl-nonori-4}) and (\ref{equ-appl-nonori-5}), we deduce that each 
$\displaystyle \int^{z}_{1} {\psi}_{j}$ is well-defined on $A(R^{1/k})$. 
Moreover $\sum_{j=1}^{4} {\psi}_{j}^{2}\equiv 0$ holds. Since $k$ is odd, we have
\begin{equation}\label{thm-appl-nonori-6}
I^{\ast}({\psi}_{j}) = \bar{\psi}_{j}, \quad j=1, \cdots , 4,  
\end{equation} 
where $I\colon A(R^{1/k})\to A(R^{1/k})$ is the lift of the previous involution in $A(R)$. 
Indeed, $I$ is represented as $I (z)=-(\bar{z})^{-1}$ here. We note that $\lim_{k\to \infty} R^{1/k}= 1$ 
and the zeros of $f$ are not on the circle $\{z \,;\, |z|=1\}$. Thus we take $k$ large enough, we can assume that 
$f$ never vanishes on the closure of $A(R^{1/k})$. Furthermore, since the only poles of $f$ are $0$ and 
$\infty$, there exists some real number $c >1$ such that 
$$
\dfrac{1}{c}< |f(z)|< c,  
$$
for any $z\in A(R^{1/k})$. Hence $\sum_{j=1}^{4}|{\psi}_{j}|^{2} \not= 0$, and if we define 
$ds^{2}_{0}= \sum_{j=1}^{4}|{\psi}_{j}|^{2}$, then we have  
$$
\dfrac{1}{c^{2}}T^{\ast}_{k} (ds^{2})\leq ds^{2}_{0} \leq c^2 T^{\ast}_{k} (ds^{2}).  
$$  
Since $ds^{2}$ is complete, the metric $T^{\ast}_{k} (ds^{2})$ and $ds^{2}_{0}$ are also complete. 

Therefore we obtain the conformal minimal immersion 
$$
X\colon A(R^{1/k})\to {\R}^{4},  \quad 
\displaystyle X(z)=\text{Re}\int^{z}_{1} ({\psi}_{1}, {\psi}_{2}, {\psi}_{3}, {\psi}_{4}) 
$$
and the induced metric $ds^{2}_{0}$ is complete and each component of the Gauss map $g_{i}\circ T_{k}$ 
($i=1, 2$) omits $4$ points in $\RC$. From (\ref{thm-appl-nonori-6}), the immersion $X$ induces a 
minimal immersion from the M\"obius strip $A(R^{1/k}) /\langle I \rangle$ to ${\R}^{4}$, and 
each component of the generalized Gauss map omits $2$ points in $\mathbf{RP}^{2}$. 
\end{proof}

\begin{remark}\label{rmk-appl-nonori-2}
From a similar argument of the proof, we can show that there exist nonflat complete nonorientable minimal surfaces 
in ${\R}^{4}$ one of which components of the generalized Gauss map is nonconstant and omits $1$ 
point in $\R\Pi^{2}$ and the other is constant. 
\end{remark}

Finally, we deal with value distribution of the generalized Gauss map of complete nonorientable minimal surfaces 
in ${\R}^{4}$ with finite total curvature. Applying \cite[Theorem 6.9]{HO1985} (see also \cite[Theorem 3.2]{Ka2009}) 
to the generalized Gauss map, we get the following: 

\begin{proposition}\label{thm-appl-nonori-3} 
Let $\widehat{X}\colon \widehat{\Sigma}\to {\R}^{4}$ be a nonflat complete nonorientable minimal surface with finite total curvature 
and $\widehat{G}=(\hat{g_{1}}, \hat{g_{2}})$ the generalized Gauss map of $\widehat{X}(\widehat{\Sigma})$. 
\begin{enumerate}
\item[(i)] Assume that $\hat{g}_{1}$ and $\hat{g}_{2}$ are both nonconstant. Then at least one of them can 
omit at most $1$ point in $\R\Pi^{2}$. 
\item[(i\hspace{-.1em}i)] If either $\hat{g}_{1}$ or $\hat{g}_{2}$, say $\hat{g}_{2}$, is constant, 
then $\hat{g}_{1}$ can omit at most $1$ point in $\R\Pi^{2}$. 
\end{enumerate}
\end{proposition} 

However we do not know whether Proposition \ref{thm-appl-nonori-3} is optimal or not. 


\end{document}